\documentclass[a4paper,11pt,leqno]{amsart}
\usepackage{amsmath}
\usepackage{amsthm}
\usepackage{amsfonts}
\usepackage{amssymb}

\usepackage{esint}

\usepackage[usenames]{color}
\usepackage[all]{xy}

\usepackage{url}
\usepackage[all]{xy}
\usepackage{graphicx}
\usepackage{latexsym}
\usepackage[cp850]{inputenc}
\usepackage{epsfig}
\usepackage{psfrag}
\usepackage{dsfont}

\newtheorem{theorem}{Theorem}[section]{\bf}{\it}
\newtheorem{lemma}[theorem]{Lemma}{\bf}{\it}
{\bf}{\it}
\newtheorem{corollary}[theorem]{Corollary}{\bf}{\it}
{\bf}{\it} 
{\bf}{\it}
\newtheorem*{theorem*}{Theorem}

\newtheorem{remark}{Remark}

\newtheorem*{namedtheorem}{\theoremname}
\newcommand{\theoremname}{testing}

{\bf}{\it}
{\bf}{\it}

\theoremstyle{remark}

\theoremstyle{definition}
\theoremstyle{remark}

\numberwithin{equation}{section}

\newcommand{\diam}{\operatorname{diam}}

\newcommand{\R}{\mathbb R}
\newcommand{\Z}{\mathbb Z}

\newcommand{\loc}{{\operatorname{loc}}}

\newcommand{\dist}{{\operatorname{dist}\,}}

\newcommand{\Q}{{\mathbb{Q}}}

\newcommand{\Ext}{{\operatorname{Ext}}}

\newdimen\vintkern\vintkern11pt
\def\vint{-\kern-\vintkern\int}

\newcommand{\bS}{\mathbb{S}}

\newcommand{\haus}{\mathcal{H}}

\newcommand{\bT}{\mathbb{T}}

\newcommand{\ord}{\mathrm{ord}}
\newcommand{\cd}{\mathrm{cd}}
\newcommand{\rank}{\mathrm{rank}}

\begin{document}

\title{Rigidity of extremal quasiregularly elliptic manifolds}
\date{\today}

\author{Rami Luisto \and Pekka Pankka}
\address{Department of Mathematics and Statistics, P.O. Box 68 (Gustaf H\"allstr\"omin katu 2b), FI-00014 University of Helsinki, Finland}
\address{Department of Mathematics and Statistics, P.O. Box 68 (Gustaf H\"allstr\"omin katu 2b), FI-00014 University of Helsinki, Finland \and Department of Mathematics and Statistics, P.O. Box 35, FI-40014 University of Jyv\"askyl\"a, Finland}

\thanks{The authors are supported by the Academy of Finland project \#256228.}

\begin{abstract}
We show that for a closed $n$-manifold $N$ admitting a quasiregular mapping from Euclidean $n$-space the following are equivalent: (1) order of growth of $\pi_1(N)$ is $n$, (2) $N$ is aspherical, and (3) $\pi_1(N)$ is virtually $\Z^n$ and torsion free.
\end{abstract}

\subjclass[2010]{30C65 (57M12)}

\maketitle

\section{Introduction}

In this note we consider closed Riemannian manifolds $N$ admitting a quasiregular mapping from $\R^n$, i.e.\;\emph{quasiregularly elliptic manifolds}. A non-constant continuous mapping $f\colon \R^n \to N$ is \emph{($K$-)quasiregular} if $f$ belongs to the Sobolev space $W^{1,n}_\loc(\R^n,N)$ and satisfies the distortion inequality
\[
|Df|^n \le K J_f \quad \textrm{a.e.\ in } \R^n,
\]
where $Df$ is the differential of the map $f$ and $J_f$ the Jacobian determinant. By Reshetnyak's theorem, quasiregular mappings are \emph{branched covers}, that is, discrete and open mappings, see e.g.\;\cite[Theorem I.4.1]{Rickman-book}. 

In dimensions $n=2,3$, closed quasiregularly elliptic manifolds are fully understood. In dimension $n=2$, the manifolds are the $2$-sphere $\bS^2$ and the $2$-torus $\bT^2$ by the uniformization theorem and Stoilow factorization. In dimension $n=3$, the possible closed targets of quasiregular mappings from $\R^3$ are $\bS^3$, $\bS^2\times \bS^1$, $\bT^3$, and their quotients. The completeness of this list follows from the geometrization theorem; see Jormakka \cite{JormakkaJ:Quamf3}. 

By Varopoulos' theorem \cite[pp. 146-147]{VaropoulosN:Anagog} \emph{the fundamental group of a quasiregularly elliptic $n$-manifold, $n\ge 2$, has polynomial growth of order at most $n$}. In particular, the fundamental group is virtually nilpotent by Gromov's theorem,
see \cite{Gromov81}.

Our main theorem is the following characterization of quasiregularly elliptic manifolds which are extremal in the sense of Varopoulos' theorem. Recall that a manifold $N$ is \emph{aspherical} if $\pi_k(N)=0$ for $k>1$. 
We denote the degree of polynomial growth of a finitely generated group $G$ by $\ord(G)$.

\begin{theorem}
\label{thm:TFAE}
Let $N$ be a closed quasiregularly elliptic $n$-manifold. Then the following are equivalent:
\begin{enumerate}
\item $\ord(\pi_1(N)) = n$, \label{item:ord}
\item $N$ is aspherical, and \label{item:aspherical}
\item $\pi_1(N)$ is virtually $\Z^n$ and torsion free. \label{item:Zn}
\end{enumerate}
\end{theorem}

Given the classification of quasiregularly elliptic manifolds in dimensions $n = 2,3$, we conclude that, in these dimensions, only manifolds $N$ satisfying \eqref{item:ord} are quotients of the $n$-torus $\bT^n$. For $n\ge 5$, by deep rigidity theorems of Hsiang and Wall \cite{HsiangW:HomtII} and Farrell and Hsiang \cite{FarrellF:Topcfa} conditions (2) and (3) imply that $N$ is homeomorphic to a quotient of $\bT^n$; see also Freedman and Quinn \cite[Section 11.5]{FreedmanM:Top4m} for the extension to $n=4$. Since the $n$-torus has a unique Lipschitz structure for $n\ge 5$ by Sullivan's theorem \cite[Theorem 2]{SullivanD:Hypgh}, we have the following corollary.
 
\begin{corollary}
\label{cor:Tn}
Let $N$ be a closed quasiregularly elliptic $n$-manifold, $n\ge 2$, with $\ord(\pi_1(N))=n$. Then there exists a (topological) covering map $\bT^n \to N$. For $n\ne 4$, there exists a locally bilipschitz covering map $\bT^n \to N$. 
\end{corollary}

Theorem \ref{thm:TFAE} can be viewed as a \emph{(branched) quasiconformal Bieberbach theorem}. By the classical Bieberbach theorem, \emph{an $n$-dimensional crystallographic group $\Gamma$ (that is, a cocompact discrete group of Euclidean isometries of $\R^n$) is virtually $\Z^n$ and torsion free}.
By a classical result of Auslander and Kuranishi \cite{AuslanderL:HolglE} all groups which are virtually $\Z^n$ and torsion free are crystallographic; see also Thurston \cite[Section 4.2]{ThurstonW:3dimgt}.

\begin{corollary}
  Let $N$ be a closed orientable Riemannian $n$-manifold, $n\ne 4$, satisfying $\ord(\pi_1(N))=n$. Then $N$ is quasiregularly elliptic if and only if $N$ is bilipschitz homeomorphic to $\R^n/\Gamma$, where $\Gamma$ is a crystallographic group.
\end{corollary}

In particular, we obtain a partial
answer to a question of Bonk and Heinonen \cite[p. 222]{BonkM:Quamc}. 
\begin{corollary}
  Let $N$ be a closed quasiregularly elliptic $n$-manifold  satisfying
  $\ord(\pi_1(N)) = n$.
  Then
  $$\dim H^\ast(N ; \Q ) \leq 2^n.$$
\end{corollary}

Finally, Theorem \ref{thm:TFAE} bears a close resemblance to a result of Gromov on elliptic manifolds \cite[Corollary 2.43]{Gromov-book}: \emph{The fundamental group of a closed aspherical elliptic manifold is virtually $\Z^n$}. A manifold $N$ is called \emph{elliptic} if there exists a Lipschitz map $\R^n \to N$ of non-zero asymptotic degree. We refer to \cite[Section 2.41]{Gromov-book} for a detailed discussion on elliptic manifolds. By Theorem \ref{thm:TFAE}, for quasiregularly elliptic manifolds, the topological assumption on asphericity can be replaced by the geometric assumption on the Euclidean volume growth of the universal cover. 

\subsection{Sketch of the proof of Theorem \ref{thm:TFAE}}

The main content of Theorem \ref{thm:TFAE} is that \eqref{item:ord} implies \eqref{item:aspherical}. Well-known arguments in cohomological group theory show that in this setting \eqref{item:aspherical} implies \eqref{item:Zn} since the growth rate of the fundamental group of a quasiregularly elliptic manifold is polynomial. Condition \eqref{item:Zn} trivially implies \eqref{item:ord}. 

The rest of the paper is devoted to the proofs of implications \eqref{item:ord}$\Rightarrow$\eqref{item:aspherical} and \eqref{item:aspherical}$\Rightarrow$\eqref{item:Zn}. In Section \ref{sec:Loewner} we discuss the observation that the universal cover $\tilde N$ of $N$ is a \emph{Loewner space} in the sense of Heinonen and Koskela \cite{HK}. The proof is based on the verification of a $(1,n)$-Poincar\'e inequality on $\tilde N$. Using a similar argument as in \cite{KleinerB:NewpGt}, we give a simple proof for a $(1,1)$-Poincar\'e inequality on $\tilde N$, which yields the required $(1,n)$-Poincar\'e inequality trivially; note that Saloff-Coste's $(2,2)$-Poincar\'e inequality for $\pi_1(N)$ in \cite{KleinerB:NewpGt} also suffices.

Using the Euclidean volume growth and the Loewner property of $\tilde N$, we show that there exists a proper map $\R^n \to \tilde N$ and that $N$ is aspherical. 
Finally, in Section \ref{sec:cgt}, we discuss the implication that asphericity of $N$ virtually detects the group $\pi_1(N)$. The argument is almost verbatim to the proof of Bridson and Gersten in \cite[Theorem 5.9]{BridsonGersten} for quasi-isometric rigidity of $\Z^n$, although our method allows us to detect the quasi-isometry type of $\pi_1(N)$ only posteriori.

\bigskip
\noindent
\textbf{Acknowledgements.} 
We would like to thank the referees for suggestions and helpful remarks.

\section{The Loewner property}
\label{sec:Loewner}

A metric measure space $(X,d,\mu)$ is \emph{$n$-Loewner} if there exists a function $\phi \colon [0,\infty) \to [0,\infty)$ such that
\begin{equation}
\label{eq:Loewner}
\operatorname{mod}_n (E,F) \geq \phi(t)
\end{equation}
whenever $E$ and $F$ are two disjoint, nondegenerate continua in $X$ and
\[
t \geq \frac{\dist(E,F)}{\min(\diam E, \diam F)}.
\]
Here $\operatorname{mod}_n (E,F)$ is the \emph{$n$-modulus} of the family $\Gamma(E,F)$ of all paths connecting $E$ and $F$, that is,
\[
\operatorname{mod}_n (E,F) = \inf \int_X \rho^n \, \mathrm{d}\mu,
\]
where the infimum is taken over all nonnegative Borel functions $\rho \colon X \to [0,\infty]$ satisfying
\[
\int_\gamma \rho \,\mathrm{d}s \geq 1
\] 
for all locally rectifiable paths $\gamma\in \Gamma(E,F)$.
 
In this section, we consider the Loewner property of universal covers $\tilde N$ of closed Riemannian manifolds $N$ satisfying $\ord(\pi_1(N))=n$. The Riemannian metric induced by the covering $\tilde N \to N$ makes $\tilde N$ into a geodesic metric space with Euclidean volume growth, that is, $\tilde N$ is \emph{Ahlfors $n$-regular}:
\[
\haus^n(B(x,r)) \approx r^n
\]
for every ball $B(x,r)$ in $\tilde N$ of radius $r>0$ about $x\in \tilde N$. 

It seems that the following theorem has not been reported in the literature although it is well-known to the experts. 

\begin{theorem}\label{theorem:Loewner}
Let $N$ be a closed and connected Riemannian $n$-manifold for $n \geq 2$ satisfying $\ord (\pi_1(N)) = n$. Then the universal cover $\tilde N$ of $N$ is $n$-Loewner.
\end{theorem}

\begin{remark}\label{remark:AhlforsRegularity}
Since $\tilde N$ is Ahlfors $n$-regular and geodesic, the Loewner property \eqref{eq:Loewner} holds with $\phi$ satisfying 
\begin{align*}
\phi(t) \approx 
  \begin{cases}
    \log \left( t \right), & \text{ when $t$ is small} \\
    \log \left( t \right)^{1-n}, & \text{ when $t$ is large};
  \end{cases}
\end{align*}
see \cite[Theorem 3.6]{HK}.
\end{remark}

By a result of Heinonen and Koskela \cite[Corollary 5.13]{HK}, $\tilde N$ is $n$-Loewner if and only if $\tilde N$ supports a weak $(1,n)$-Poincar\'e inequality: there exists $C>0$ and $\lambda\ge 1$ for which 
\begin{align}\label{eq:Poincare1n}
\fint_{B(x,r)} |f-f_{B(x,r)}| \leq  Cr \left( \fint_{ B(x,\lambda r)}|\nabla f |^n \right)^{\frac{1}{n}}
\end{align}
for all $x \in \tilde N$, $r > 0$ and $f \in C^\infty(\tilde N)$. 
Here we denote by $f_{B(x,r)}$ the average
\begin{align*}
  \fint_{B(x,r)} f(y) \, \mathrm{d}y.
\end{align*}

The weak $(1,n)$-Poincar\'e inequality follows directly from Kleiner's weak $(2,2)$-Poincar\'e inequality in \cite[Theorem 2.2.]{KleinerB:NewpGt} by H\"older's inequality. Laurent Saloff-Coste's argument in \cite{KleinerB:NewpGt} can, however, be used to show also that $\tilde N$ satisfies a weak $(1,1)$-Poincar\'e inequality. Since we have not seen this explicitly stated in the literature, we give a short proof based on \cite[Theorem 2.2]{KleinerB:NewpGt} and \cite[p. 308-309]{CoulhonSaloff-Coste93} for the reader's convenience.
\begin{theorem}
\label{thm:1_1_tilde_M}
Let $N$ be a closed Riemannian $n$-manifold with a polynomially growing fundamental group. Let $\tilde N$ be the universal cover of $N$. Then there exists $C>0$ so that
\begin{align}\label{eq:Poincare1n}
  \fint_{B(x,r)} |f- f_{B(x,r)}|  \leq  Cr \fint_{ B(x,3r)}|\nabla f|
\end{align}
for all $x \in \tilde N$, $r > 0$, and $f \in C^\infty(\tilde N)$.
\end{theorem}

The proof is based on the following lemma. In the statement, we assume that we have chosen a fixed finite (symmetric) generating set $S$ for a group $\Gamma$, and denote by 
\[
B(x,r)=\{ x s_1\cdots s_r \in \Gamma \colon s_i \in S\ \mathrm{for\ } i=1,\ldots, r\}
\]
the ball of radius $r>0$ about $x$ in $\Gamma$, $V(r) = \#B(e,r)$, 
\begin{align*}
  f_{B(x,r)} := \frac{1}{V(r)}\sum_{y \in B(x,r)} f(y)
  \quad \text{ and } \quad
  \nabla f (y) := \sum_{z \in yS} \left| f(z) - f(y) \right|.
\end{align*}

\begin{lemma}\label{lemma:GroupHasPoincare}
Let $\Gamma$ be a finitely generated group. Then 
\begin{align}\label{eq:PoincareGroup}
\sum_{y \in B(x,r)} | f(y) - f_{B(x,r)} | \leq (r+1) \frac{V(2r)}{V(r)} \sum_{y \in B(x,3 r)} | \nabla f(y)|
\end{align}
for all $r > 0$, $x \in \Gamma$, and $f \colon \Gamma \to \R$.
\end{lemma}

\begin{proof}
We fix for each $g \in G$ a geodesic $\gamma_g \colon \{ 0 , \ldots , |g| \} \to \Gamma$ in $\Gamma$ connecting the neutral element $e$ of $\Gamma$ to $g$.

For any $r > 0$ and $ y \in B(x,r)$ the mapping 
\begin{align*}
B(x,r) \times \{ 0, \ldots , |y|\} \to G, \quad (x,i) \mapsto x \gamma_y(i),
\end{align*}
is at most $(r+1)$-to-$1$. 

For every $y_0 \in B(x,r)$, we have 
\begin{align*}
V(r) &\sum_{y \in B(x,r)} |f(y) - f_{B(x,r)}(x)| \leq \sum_{y,z \in B(x,r)} |f(y) - f(z)| \\
&\leq \sum_{y \in B(x,r)} \sum_{w \in B(x,2r)} |f(yw) - f(z)| \\
&\leq \sum_{y \in B(x,r)} \sum_{w \in B(x,2r)} \sum_{i = 0}^{\ell(\gamma_w)-1} |f(y\gamma_w^{i+1}) - f(y\gamma_w^{i})| \\
&\leq \sum_{w \in B(x,2r)} \sum_{i = 0}^{\ell(\gamma_w)-1} | \nabla f (y_0 \gamma_w^i) |.
\end{align*}
On the other hand,
\begin{align*}  
\sum_{w \in B(x,2r)} \sum_{i = 0}^{\ell(\gamma_w)} | \nabla f (y\gamma_w^i) |
&\leq \sum_{w \in B(x,2r)} \sum_{z \in B(x,3r)} (r+1) | \nabla f (z)| \\
&\leq V(2r) (r+1) \sum_{z \in B(x,3r)} | \nabla f (z)|.
\end{align*}
This proves the claim.  
\end{proof}

\begin{proof}[Proof of Theorem \ref{thm:1_1_tilde_M}]
For every $r>0$, the universal cover $\tilde N$ satisfies the \emph{local} $(1,1)$-Poincar\'e inequality
\begin{align}\label{eq:PoincareLocal}
  \int_{B(x,r)} | f(y) - f_{B(x,r)} | \, dy
  \leq  \beta
  \int_{B(x, r)}|\nabla f|
\end{align}
for $f\in C^\infty(\tilde N)$, where $\beta$ depends on $r$; see Kanai \cite[Lemma 8]{Kanai}.

Since $\pi_1(N)$ has polynomial growth, the ratio $V(2r)/V(r)$ is uniformly bounded for all $x\in \pi_1(N)$ and $r>1$. Thus, by Lemma \ref{lemma:GroupHasPoincare}, there exists $C>1$ so that 
\[
\sum_{y \in B(x,r)} | f(y) - f_{B(x,r)} | \leq C r \sum_{y \in B(x,3 r)} | \nabla f(y)|
\]
for all $r > 0$, $x \in \Gamma$ and $f \colon \Gamma \to \R$.

The claim now follows from \cite[Th\'eor\`eme 7.2.(3)]{CoulhonSaloff-Coste95}.
\end{proof}

This concludes the proof of Theorem \ref{theorem:Loewner}.

\section{Euclidean volume growth and asphericity}
\label{sec:asphericity}

In this section we show that a closed quasiregularly elliptic manifold with maximally growing fundamental group is aspherical. We obtain this result by combining the following lemmas.

\begin{lemma}
\label{lemma:ExistsAProperMap}
Suppose $N$ is a closed quasiregularly elliptic $n$-manifold, 
$n\ge 2$, with $\ord(\pi_1(N))=n$. Then there exists a 
proper quasiregular map $\R^n \to \tilde N$ 
into the universal cover $\tilde N$ of $N$.
\end{lemma} 
\begin{proof}
By Zalcman's lemma (see Bonk--Heinonen \cite[Corollary 2.2]{BonkM:Quamc}), there exists a uniformly locally H\"older continuous quasiregular map 
$f\colon \R^n \to N$ satisfying
\begin{equation}
\label{eq:J}
\int_{B^n(x,r)} J_f \le C r^n
\end{equation}
where $C>0$ is a constant independent of the ball
$B^n(x,r)\subset \R^n$. 
Let $\tilde f$ be a lift of $f$ to the universal cover 
$\tilde N$ of $N$. Since the universal covering map is a 
local isometry, $\tilde f$ again satisfies inequality 
\eqref{eq:J}. 

Since $\tilde N$ is Ahlfors $n$-regular and $n$-Loewner, we have that, for every $a \in \tilde N$, $d(\tilde f(x),a) \to \infty$ as $|x| \to \infty$ by the Onninen--Rajala theorem \cite[Theorem 12.1]{OnninenRajala}. Indeed, (a) in \cite[Theorem 12.1]{OnninenRajala} can be replaced with \eqref{eq:J}; see Lemma 12.13 and (12.3) in \cite{OnninenRajala}. Thus $\tilde f$ is a proper map.
\end{proof}

\begin{lemma}
\label{lemma:asphericity}
Let $N$ be a connected, simply connected, and oriented Riemannian $n$-manifold. Then $N$ is aspherical if there exists a proper branched cover $\R^n \to \tilde  N$ into the universal cover $\tilde N$ of $N$.
\end{lemma}

\begin{proof}
  Suppose there exists $k\ge 2$ for which $\pi_k(\tilde N)\ne 0$. Let $k\ge 2$ be the smallest such index.
  Then, by the Hurewicz isomorphism theorem, $\pi_k(\tilde N)$ is isomorphic to $H_k(\tilde N)$.

If $H_k(\tilde N)$ is free abelian, the universal coefficient theorem for cohomology immediately yields $H^k(\tilde N,\Z)\ne 0$. If $H_k(\tilde N)$ is not free, $\Ext(H_k(\tilde N),\Z)\ne 0$. Thus $H^{k+1}(\tilde N,\Z)\ne 0$ by the universal coefficient theorem.

We conclude that there exists an index $\ell \geq 2$ such that
  $H^\ell(\tilde N, \Z) \neq 0$. Since $N$ receives a proper branched cover from $\R^n$
  the fundamental group $\pi_1(N)$ must be infinite. Thus the universal cover $\tilde N$ is unbounded,
  so $H_c^0(\tilde N, \Z) = 0 $ and by the Poincar\'e duality $H^n(\tilde N, \Z) = 0$. In particular, $\ell<n$.

Let $f\colon \R^n \to \tilde N$ be a proper branched cover.
Then $f^* \colon H^n_c(\tilde N;\Z) \to H^n_c(\R^n;\Z)$ is non-trivial.

By the Poincar\'e duality, there exists
  $c  \in H^\ell(\tilde N,\Z)$   
  and 
  $c' \in H^{n-\ell}_c(\tilde N,\Z)$ 
  satisfying 
  $c\cup c'\ne 0\in H^n_c(\tilde N, \Z)$. 
  Then 
  \begin{align*}
    f^*c \cup f^*c' = f^*(c\cup c')\ne 0 \in H^n_c(\R^n,\Z)
  \end{align*}
This is a contradiction, since $f^*c= 0$. Thus $\pi_k(N)=0$ for all $k\ge 2$.
\end{proof}

\section{Closed aspherical quasiregularly elliptic manifolds}
\label{sec:cgt}

In this section we show that
\eqref{item:aspherical} implies \eqref{item:Zn} 
in the setting of Theorem \ref{thm:TFAE}; 
see also \cite[Corollary 2.43]{Gromov-book}.
The argument of the proof is almost identical to the proof of theorem \cite[Theorem 5.9.]{BridsonGersten}:
\emph{a finitely generated group that is quasi-isometric to the group $\Z^n$ is virtually $\Z^n$}.
Due to the differences in the setting, we give a proof for the reader's convenience.
\begin{lemma}\label{lemma:last}
\label{lemma:virt_Z}
Let $N$ be a closed, connected, aspherical $n$-manifold with $\ord (\pi_1(N))\le n$. Then $\pi_1(N)$ is virtually $\Z^n$ and torsion free.
\end{lemma}

\begin{proof}
Since $N$ is aspherical, the cohomological dimension $\cd(\pi_1(N))$ of $\pi_1(N)$ is at most $n$ by \cite[Prop VIII.2.2]{BrownK:Cohg}. Thus $\pi_1(N)$ has finite cohomological dimension and is torsion free by \cite[Prop VIII.2.8]{BrownK:Cohg}. 

Since $\pi_1(N)$ has polynomial growth, it is virtually nilpotent. Let $\hat N$ be a finite cover of $N$ having nilpotent fundamental group. Then $\pi_1(\hat N)$ has a lower central series
\[
\pi_1(\hat N) = \Gamma_0 \supset \Gamma_1 \supset \cdots \supset \Gamma_m = \{1\}
\]
with free abelian quotients $\Gamma_{k-1}/\Gamma_k$ for $k=1,\ldots, m$. Moreover, since $\pi_1(\hat N)$ is finitely generated, the cohomological dimension coincides with the Hirsh number of $\pi_1(\hat N)$, that is, 
\begin{equation}
\label{eq:Hirsh}
\cd(\pi_1(\hat N)) = \sum_{k=1}^m \rank \left( \Gamma_{k-1}/\Gamma_k\right).
\end{equation}
We refer to \cite[Section 8.8]{GruenbergK:Cohtgt} (especially Theorem 5) for these details.

On the other hand, by the growth formula for nilpotent groups (see Bass \cite{BassH:Degpgf}), we have
\begin{equation}
\label{eq:Bass2}
\ord(\pi_1(\hat N)) = \sum_{k=1}^m k\;\rank \left(\Gamma_{k-1}/\Gamma_k\right).
\end{equation} 
Since $\cd(\hat N)=n$ by \cite[Theorem VIII.8.1]{BrownK:Cohg}, we have 
\begin{equation}
\label{eq:equality}
\sum_{k=1}^m k\;\rank \left(\Gamma_{k-1}/\Gamma_k\right) = 
\ord(\pi_1(\hat N)) \le \cd(\pi_1(\hat N)) = \sum_{k=1}^m \rank \left( \Gamma_{k-1}/\Gamma_k\right)
\end{equation}
by combining \eqref{eq:Hirsh}, \eqref{eq:Bass2}, and \eqref{eq:equality}. Thus
\[
\rank \left( \Gamma_{k-1}/\Gamma_k \right) = 0
\]
for all $k=2,\ldots, m$ and $\pi_1(\hat N)$ is abelian. Thus $\pi_1(N)$ is virtually $\Z^n$.
\end{proof}



\def\cprime{$'$}

\end{document}